\documentclass[12pt, oneside]{amsart}   	
\usepackage{geometry}                		
\usepackage{graphicx}				
\usepackage{amssymb}
\usepackage{amsthm}
\newtheorem{theorem}{Theorem}[section]
\newtheorem{lemma}[theorem]{Lemma}
\newtheorem{corollary}[theorem]{Corollary}
\numberwithin{equation}{section}

\newcommand{\br}{\mathbf {R}}
\newcommand{\h}{\frak{h}_n}
\newcommand{\pari}{\frac{\partial}{\partial x^i}}
\newcommand{\parj}{\frac{\partial}{\partial x^j}} 
\newcommand{\son}{\frak{so}(n)}
\newcommand{\lra}{\longrightarrow}
\newcommand{\fhn}{\frak{h}_{n}}
\newcommand{\p}{\partial}
\newcommand{\g}{\gamma^*_n}

\title{Leibniz Cohomology and Connections on Differentiable Manifolds}
\author{Jerry M. Lodder}
\thanks{Mathematical Sciences, Deptartment 3MB; New Mexico State University; Las Cruces, NM, 88003;
USA, \tt{jlodder@nmsu.edu}}

\begin{document}
\maketitle

\noindent
{\bf{Abstract.}}  We show how an affine connection on a Riemannian manifold occurs naturally as a cochain
in the complex for Leibniz cohomology of vector fields with coefficients in the adjoint representation.  The Leibniz 
coboundary of the Levi-Civita connection can be expressed as a sum of two terms, one the Laplace-Beltrami 
operator and the other a Ricci curvature term.  The vanishing of this coboundary has an interpretation in terms 
of eigenfunctions of the Laplacian.  Additionally, we compute the Leibniz cohomology with adjoint coefficients 
for a certain family of vector fields on Euclidean $\br^n$ corresponding to the affine orthogonal Lie algebra, $n \geq 3$.



\section{Introduction}
We study Leibniz cohomology, $HL^*$, of a Lie algebra $\frak{g}$ with coefficients in its adjoint representation, 
written $HL^* (\frak{g}; \, \frak{g})$, as a setting for capturing an affine connection \cite[Chapter 2]{doCarmo} 
as a cochain in the Leibniz complex  
$CL^* ( \frak{g}; \, \frak{g}) = {\rm{Hom}}_{\br}(\frak{g}^{\otimes *}, \, \frak{g})$.  
More specifically, let $M$ be a (real) differentiable manifold and $C^{\infty}(M)$ the ring of real-valued differentiable functions 
on $M$.  Let $\chi (M)$ be the Lie algebra of differentiable vector fields on $M$.  An affine or Koszul connection 
$\nabla$ on $M$ is an $\br$-linear map
$$  \nabla : \chi (M) \underset{\br}{\otimes} \chi (M) \to \chi (M)  $$
that is $C^{\infty}(M)$-linear in the first (left) tensor factor, but a derivation on $C^{\infty}(M)$ functions in the second (right) 
tensor factor.   For $X$, $Y \in \chi (M)$, we write $\nabla (X \otimes Y) = \nabla_X (Y)$.  For $f \in C^{\infty}(M)$, we have
\begin{equation} \label{nabla}
\nabla_X (f Y) = X(f) \, Y + f \, \nabla_X (Y) .  
\end{equation}
The failure of $\nabla$ to be $C^{\infty}(M)$ linear in its second $\chi (M)$ factor precludes $\nabla$ from being a two-tensor. 
In $\S$ 2 we interpret $\delta \nabla$, the Leibniz coboundary of $\nabla$, in terms of the Laplace-Beltrami operator 
and the curvature operator $R : \chi (M)^{\otimes 3} \to \chi (M)$.  The vanishing of $\delta \nabla$ has an 
interpretation in terms of eigenfunctions of the Laplacian on manifolds of constant scalar curvature,
Theorem \eqref{eigenfunction}.

Leibniz cohomology with coefficients in the adjoint representation for semi-simple Leibniz algebras has been studied in 
\cite{Feldvoss-Wagemann}, while $HL^*( \frak{L}; \, \frak{L})$ has been used to study deformation theory 
in the category of Leibniz algebras \cite{Fil-Mandal}, \cite{Fil-Mandal-Muk}.   Recall that a left Leibniz algebra
$\frak{L}$ over a commutative ring $k$ is a $k$-module equipped with a bilinear operation
$\circ: \frak{L} \times \frak{L} \to \frak{L}$ that is a derivation from the left:
$$  x \circ ( y \circ z ) = ( x \circ y ) \circ z + y \circ ( x \circ z ), \ \ \ x, \, y, \, z \in \frak{L} .  $$
There is a similar notion for a right Leibniz algebra \cite{Loday-Pir}.  A Lie algebra $\frak{g}$ with its
skew-symmetric bracket is both a left and a right Leibniz algebra, with the derivation property subsumed in
the Jacobi identity.   When working with Lie algebras in this paper, we will denote the Lie bracket as $[ \ , \ ]$ 
instead of $\circ$.  Leibniz algebras were called $D$-algebras by Bloh \cite{Bloh} and later rediscovered and 
studied  extensively by Loday \cite[\S 10.6.1]{Loday1992}.

Since the bracket of a Lie algebra $\frak{g}$ is skew-symmetric, there is more flexibility in stating  the cochain complex
yielding  $HL^*(\frak{g} ; \, \frak{g} )$ than for a  (left or right) Leibniz algebra \cite{Feldvoss-Wagemann} \cite{Loday-Pir}.  
Let
$$  CL^n( \frak{g} ; \, \frak{g} ) = {\rm{Hom}}_k ( \frak{g}^{\otimes n}, \, \frak{g} ), \ \ \ n \geq 0,  $$
where $\frak{g}$ is a Lie algebra over the commutative ring $k$.  For calculations in this paper, $k = \br$.  
Let $\frak{g}$ act on itself via the adjoint representation.  For $g \in \frak{g}$, the linear map
${\rm{ad}}_g : \frak{g} \to \frak{g}$ is given by
$$  g \cdot x :=  {\rm{ad}}_g (x) = [g, \, x], \ \ \ x \in \frak{g}.  $$
Then $HL^*(\frak{g}; \, \frak{g})$ is the cohomology of the cochain complex $CL^*(\frak{g}; \, \frak{g} )$:
\begin{equation}  \label{HL-cochains}
\begin{split} 
& {\rm{Hom}}_k (k, \, \frak{g}) \overset{\delta}{\longrightarrow} {\rm{Hom}}_k (\frak{g}, \, \frak{g}) 
\overset{\delta}{\longrightarrow} {\rm{Hom}}_k (\frak{g}^{\otimes 2}, \, \frak{g}) \overset{\delta}{\longrightarrow} 
{\rm{Hom}}_k (\frak{g}^{\otimes 3}, \, \frak{g}) \overset{\delta}{\longrightarrow}  \\
& \ldots \ \overset{\delta}{\longrightarrow} {\rm{Hom}}_k( \frak{g}^{\otimes n}, \, \frak{g})
\overset{\delta}{\longrightarrow} {\rm{Hom}}_k( \frak{g}^{\otimes (n+1)}, \, \frak{g})
\overset{\delta}{\longrightarrow} \ldots  \  .
\end{split}
\end{equation}
For $f : \frak{g}^{\otimes n} \to \frak{g}$, the coboundary $\delta f : \frak{g}^{\otimes (n+1)} \to \frak{g}$ is given by
\begin{equation} \label{coboundary}
\begin{split}
(\delta f)&(g_1 \otimes g_2 \otimes \ldots \otimes g_{n+1}) =  \\
& \sum_{i=1}^{n+1} (-1)^i g_i \cdot f (g_1 \otimes g_2 \otimes \ldots \widehat{g}_i  \ldots \otimes g_{n+1} )  \\
& + \sum_{1 \leq i < j \leq n+1} (-1)^j f(g_1 \otimes \ldots \otimes g_{i-1} \otimes [g_i,  \, g_j] \otimes g_{i+1} 
\otimes \ldots \widehat{g}_j \ldots \otimes g_{n+1}).  
\end{split}
\end{equation}

In the next section we study $\delta \nabla$ for the Levi-Civita connection $\nabla$ on a differentiable manifold.  
In $\S$ 3 we offer a calculation of $HL^*(\h ; \, \h)$, where $\h$ is a certain family of vector fields on $\br^n$ 
isomorphic to the affine orthogonal Lie algebra, $n \geq 3$.  This last section contains a subsection on
the Pirashvili spectral sequence needed for the calculation of $HL^*$ as well as a subsection for the Lie-algebra 
cohomology of $\h$ with adjoint coefficients.

\section{The Coboundary of the Levi-Civita Connection}

Let $M$ be an $n$-dimensional Riemannian manifold, $U \subseteq M$ an open subset of $M$, 
and $x : \br^n \overset{\cong}{\longrightarrow} U$ a chart in the $C^{\infty}$-structure on $M$.  
Then $\pari$, $i = 1$, 2, $\ldots \, $, $n$ are
local vector fields on $U$ with $[\pari, \, \parj] = 0$.  Let $\chi(M)$ be the Lie algebra of all 
(differentiable) vector fields on $M$.  The Levi-Civita connection is the unique affine connection on $M$ that is 
symmetric and compatible with the Riemannian metric \cite[Theorem 3.6]{doCarmo}.  For a symmetric connection 
$\nabla$ and $X$, $Y \in \chi(M)$,
we have $\nabla_X(Y) - \nabla_Y(X) = [X, \, Y]$, and it follows that $\nabla_{\pari} (\parj ) = \nabla_{\parj} (\pari)$.
Let $\nabla(X \otimes Y) = \nabla_X (Y)$.  In general, for any affine connection on $M$ and $X_1$, $X_2$, 
$X_3 \in \chi(M)$, from Equation \eqref{coboundary}, we have 
\begin{equation} \label{deltaDelta}
\begin{split}
\delta \nabla  (X_1 \otimes X_2 \otimes X_3 )  = 
& - [X_1, \, \nabla_{X_2}(X_3)] + [X_2, \, \nabla_{X_1}(X_3)] - [X_3, \, \nabla_{X_1}(X_2)]  \\
& + \nabla_{[X_1, \, X_2]} (X_3) - \nabla_{[X_1, \, X_3]}(X_2) - \nabla_{X_1}([X_2, \, X_3]) .
\end{split}
\end{equation}

\begin{lemma}  \label{R-1-delta}
On Euclidean $\br^1$ let $\frac{d}{dx}$ be the canonical unit vector field for the identity chart.  Let
$X_1 = f_1 \frac{d}{dx} $, $X_2 = f_2  \frac{d}{dx}$, $X_3 = f_3 \frac{d}{dx}$ be vector fields, where
$f_1$, $f_2$, $f_3 : \br \to \br$ are $C^{\infty}$ functions.  Let $\nabla$ be the Levi-Civita connection.  Then
\begin{equation}
\delta \nabla \Big( f_1 \frac{d}{dx} \otimes f_2 \frac{d}{dx} \otimes f_3 \frac{d}{dx} \Big) =
\Big( f'_1 f_2  f'_3 - f_1 f'_2 f'_3 - f_1f_2 f''_3 \Big) \frac{d}{dx} ,
\end{equation}
where the latter can be interpreted as a certain alternating sum of terms in $(f_1 f_2 f'_3)'$.  
\end{lemma}
\begin{proof}
In general $[ f {\frac{d}{dx}}, \, g  {\frac{d}{dx}}] = (fg' - gf')  {\frac{d}{dx}}$, and 
on Euclidean space
$$  \nabla_{f \tfrac{d}{dx}} (g \tfrac{d}{dx}) = fg' \tfrac{d}{dx}.  $$
From Equation \eqref{deltaDelta}, we have
\begin{equation*}
\begin{split}
\delta & \nabla  \big( f_1 \tfrac{d}{dx} \otimes f_2 \tfrac{d}{dx} \otimes f_3 \tfrac{d}{dx} \big) = 
 - [f_1 \tfrac{d}{dx}, \, f_2 f_3' \tfrac{d}{dx}] + [f_2 \tfrac{d}{dx}, \, f_1 f_3' \tfrac{d}{dx}] - 
[f_3 \tfrac{d}{dx}, \, f_1 f_2' \tfrac{d}{dx}]  \\
& + \nabla_{(f_1f_2' - f_2 f_1') \frac{d}{dx}} (f_3 \tfrac{d}{dx}) 
- \nabla_{(f_1f_3' - f_3 f_1') \frac{d}{dx}} (f_2 \tfrac{d}{dx})
- \nabla_{f_1 \frac{d}{dx}} (f_2 f_3' - f_3 f_2') \tfrac{d}{dx}.   
\end{split}
\end{equation*}
The result now follows by combining terms.  
\end{proof}
Thus, on $\br^1$, $\delta \nabla( \frac{d}{dx}\otimes \frac{d}{dx} \otimes f \frac{d}{dx} ) = -f'' \frac{d}{dx}$.  Moving to 
$\br^2$ or to $\br^n$ in general, we recover the Laplacian of $f$ when used as a coefficient function on the third vector field.

\begin{lemma}
On Euclidean $\br^2$, let $\frac{\partial}{\partial x} = \frac{\partial}{\partial x^1}$ and 
$\frac{\partial}{\partial y} = \frac{\partial}{\partial x^2}$ for the identity chart.  Let $\nabla$ be the Levi-Civita  connection
and let $Z = c_1 \frac{\partial}{\partial x} + c_2  \frac{\partial}{\partial y}$, where $c_1$, $c_2 \in \br$.  Then
$$  \delta \nabla \big( ( \tfrac{\partial}{\partial x} \otimes \tfrac{\partial}{\partial x} + \tfrac{\partial}{\partial y} \otimes
\tfrac{\partial}{\partial y}) \otimes f Z \big) = - (f_{xx} + f_{yy}) Z .  $$
\end{lemma}
\begin{proof}
The proof follows from Lemma \eqref{R-1-delta} or a direct calculation from the definition of $\delta$.
\end{proof}

\begin{corollary}
For $Z = c_1 \frac{\p}{\p x} + c_2 \frac{\p}{\p y}$ a constant unit vector field on $\br^2$, i.e., 
$c_1^2 + c_2^2 =1$, we have
$$  \big{\langle}   \delta \nabla \big( ( \tfrac{\partial}{\partial x} \otimes \tfrac{\partial}{\partial x} + \tfrac{\partial}{\partial y} \otimes
\tfrac{\partial}{\partial y}) \otimes f Z \big), \ Z \big{\rangle} = - (f_{xx} + f_{yy}),  $$
where $\langle \, , \,  \rangle$ denotes the (Euclidean) metric on $\br^2$.
\end{corollary}
The above lemma and corollary clearly generalize to $\br^n$.  

For an arbitrary Riemannian manifold $M$, the $(1, \, 3)$ curvature tensor of an affine connection $\nabla$ is given by
$R : \chi (M)^{\otimes 3} \to \chi (M)$, where for $X$, $Y$, $Z \in \chi (M)$,
$$  R(X \otimes Y \otimes Z) = \nabla_X \nabla_Y (Z) - \nabla_Y \nabla_X (Z) 
- \nabla_{[X, \, Y]}(Z), $$
using the sign convention from Sakai \cite[p. 33]{Sakai}.  We first state a general result that applies to any symmetric 
connection on a differentiable manifold.

\begin{lemma}
Let $M$ be a differentiable manifold with a symmetric connection $\nabla$.  For $X$, $Z \in \chi (M)$, 
we have
$$  \delta \nabla ( X \otimes X \otimes Z) = 
R(X \otimes Z \otimes X) - \nabla_{X} \nabla_{X} (Z) + 
 \nabla_{\nabla_{X} (X)} (Z).  $$
\end{lemma}

\begin{proof}
 \begin{equation*}
 \begin{split} 
 \delta \nabla  & ( X \otimes X \otimes Z)  = 
- [Z, \, \nabla_{X}  (X)] -  \nabla_{ [ X, \, Z] } (X) - \nabla_X ( [ X, \, Z] )  \\
& =  - \nabla_Z \nabla_X (X)  + \nabla_{\nabla_{X} (X)} (Z) - \nabla_{[ X, \, Z]} (X) 
- \nabla_X \nabla_X (Z) + \nabla_X \nabla_Z (X)   \\
& = \nabla_X \nabla_Z (X) - \nabla_Z  \nabla_X (X)  - \nabla_{[ X, \, Z ]} ( X)
 - \nabla_X \nabla_X (Z) + \nabla_{\nabla_{X} (X)} (Z)  \\
& = R( X \otimes Z \otimes X ) - \nabla_X \nabla_X (Z) 
+ \nabla_{\nabla_{X} (X)} (Z).  
\end{split}
\end{equation*}
\end{proof}

Since a connection is a local operator, we may express $\delta \nabla$ in terms of local coordinates, 
where an expression for the Laplacian may be more identifiable.  
\begin{corollary} \label{Ricci-Laplacian}
Let $(x, \, U)$ be a coordinate chart for the $n$-dimensional Riemannian manifold $M$, let $Z \in \chi (M)$,
and let $\nabla$ be the Levi-Civita connection on $M$. Then
\begin{equation*}
\begin{split}
\delta \nabla & \Big( \sum_{i=1}^n \tfrac{\partial}{\partial x^i} \otimes \tfrac{\partial}{\partial x^i} \otimes Z \Big) = \\
& \sum_{i=1}^n  R \big(\tfrac{\partial}{\partial x^i} \otimes Z \otimes \tfrac{\partial}{\partial x^i} \big)
+ \sum_{i=1}^n \big( - \nabla_{\frac{\partial}{\partial x^i}} \nabla_{\frac{\partial}{\partial x^i} } (Z) 
+  \nabla_{\nabla_{ \frac{\partial}{\partial x^i} }(\frac{\partial}{\partial x^i})}(Z) \big). 
\end{split}
\end{equation*}
\end{corollary}
Now,  $\Delta (Z):= \sum_{i=1}^n ( - \nabla_{\frac{\partial}{\partial x^i}} \nabla_{\frac{\partial}{\partial x^i} } (Z) 
+  \nabla_{\nabla_{ \frac{\partial}{\partial x^i} }(\frac{\partial}{\partial x^i})}(Z) )$ is the Laplace-Beltrami operator 
applied to $Z$.  We interpret $R$ in terms of Ricci curvature.  
Let $\langle \ \rangle$ denote the metric on $M$, let $p \in M$, and let $z_p$, $y_p \in T_p (M)$ be tangent
vectors to $M$ at $p$.  For an orthonormal basis $\{ e_i \}$, $i = 1$, 2, $\ldots \, $, $n$, of $T_p (M)$, the
Ricci $(0, \, 2)$ tensor is given by \cite[p. 44]{Sakai}
$$  {\rm{Ric}}_p (z_p , \, y_p ) = \sum_{i = 1}^n \langle R(e_i, \, z_p, \, y_p), \, e_i \rangle =
- \sum_{i = 1}^n \langle R(e_i, \, z_p, \, e_i), \, y_p \rangle .  $$

\begin{lemma}  \label{Ricci=L(Z)}
Let $(x, \, U)$ be a chart for the $n$-dimensional Riemannian manifold $(M, \, \langle \ \rangle)$ such that
$e_i = \frac{\p}{\p x^i} \vert_p$, $ i = 1, \, 2, \, \ldots  \, , \,$ $n$,  
is an orthonormal basis for $T_p (M)$, $p \in U$.  Let $Z \in \chi (M)$, and
let $\nabla$ denote the Levi-Civita connection on $M$.  If
$$  \delta \nabla \big( \sum_{i=1}^n \tfrac{\p}{\p x^i} \otimes \tfrac{\p}{\p x^i} \otimes Z \big) = 0  $$ 
on $U$, then
$$   {\rm{Ric}}_p (Z, \, y) = \langle \Delta (Z), \, y \rangle_p  $$   
for all $y \in T_p (M)$.
\end{lemma}  
\begin{proof}
Since  $\delta \nabla ( \sum_{i=1}^n \frac{\p}{\p x^i} \otimes \frac{\p}{\p x^i} \otimes Z ) = 0$, we have
$ - \sum_{i=1}^n R ( \frac{\p}{\p x^i} \otimes Z \otimes \frac{\p}{\p x^i} ) = \Delta (Z)$.  
Thus, for any $y \in T_p (M)$, we have
$$ - \sum_{i=1}^n \langle R ( \tfrac{\p}{\p x^i} \otimes Z \otimes \tfrac{\p}{\p x^i} ), \, y \rangle_p = 
\langle \Delta (Z), \, y \rangle_p .  $$
Since $\{ \frac{\p}{\p x^i} \vert_p \}$ is an orthonormal basis for $T_p (M)$, we have
${\rm{Ric}}_p (Z, \, y) = \langle \Delta (Z), \, y \rangle_p$ at the point $p$.
\end{proof}

\begin{theorem}  \label{eigenfunction}
With $(M, \, \langle \  \rangle)$ and $\nabla$ as in Lemma \eqref{Ricci=L(Z)}, suppose that  $(x, \, U)$ 
is a geodesic normal coordinate chart around $p \in U$ with
$$  \delta \nabla \big( \sum_{i=1}^n \tfrac{\p}{\p x^i} \otimes \tfrac{\p}{\p x^i} \otimes f \tfrac{\p}{\p x^j} \big) = 0  $$ 
on $U$ for $j = 1$, $2$, $\ldots \, $, $n$, and a given $f \in C^{\infty} (M)$.  Then
$ \Delta_p f = (\frac{s}{n}) f(p)$, where $s$ denotes the scalar curvature of $M$, and $\Delta$ is the 
Laplace-Beltrami operator (applied to $f$).   
\end{theorem}

\begin{proof}
In geodesic normal coordinates $\langle \frac{\p}{\p x^i}, \, \frac{\p}{\p x^j} \rangle_p = \delta_{ij}$, and 
$\nabla_{\frac{\p}{\p x^i}\vert_p} ( \frac{\p}{\p x^j} ) = 0$.  Choose $y = \frac{\p}{\p x^j} \vert_p$ in Lemma
\eqref{Ricci=L(Z)}. Then
$$  - \sum_{i=1}^n \langle R( \tfrac{\p}{\p x^i} \otimes f \tfrac{\p}{\p x^j }\otimes \tfrac{\p}{\p x^i}), \, 
   \tfrac{\p}{\p x^j} \rangle_p = \langle \Delta (f \tfrac{\p}{\p x^j}), \, \tfrac{\p}{\p x^j} \rangle_p .  $$
In normal coordinates, 
$\Delta_p (f \frac{\p}{\p x^j}) = \sum_{i=1}^n ( \frac{\p^2 f}{\p x^i \, \p x^i} )_p \frac{\p}{\p x^j}$.  
Thus, $\langle \Delta (f \frac{\p}{\p x^j}), \, \frac{\p}{\p x^j} \rangle_p = \Delta_p (f)$.  
Since the curvature tensor $R$ is linear over $C^{\infty}(M)$, we have
$$  -f \sum_{i=1}^n \langle R ( \tfrac{p}{\p x^i} \otimes \tfrac{\p}{\p x^j} \otimes \tfrac{\p}{\p x^i} ), \,
   \tfrac{\p}{\p x^j} \rangle_p = \Delta_p (f) .  $$
It follows that
\begin{equation}
\begin{split}
& f (p) \cdot  {\rm{Ric}}_p ( \tfrac{\p}{\p x^j}, \, \tfrac{\p}{\p x^j} ) = \Delta_p (f)  \\
& f (p) \cdot s_p = f(p) \sum_{j=1}^n {\rm{Ric}}_p ( \tfrac{\p}{\p x^j}, \, \tfrac{\p}{\p x^j} ) = n  \cdot \Delta_p (f) \\
& \Delta_p (f) = ( \tfrac{s_p}{n} ) f(p) .  
\end{split}
\end{equation} 
If the above holds for all $p \in M$, then globally $\Delta (f) = (\frac{s}{n}) f$.  If furthermore $M$ is a manifold 
of constant scalar curvature $s_p \equiv c$, then $\Delta (f) = (\frac{c}{n}) f$, and $f$ is an eigenfunction of the 
Laplacian with eigenvalue $\frac{c}{n}$.  
\end{proof}

\section{Leibniz Cohomology of the Affine Orthogonal Lie Algebra}

We offer a calculation of $HL^*(\h ; \, \h)$, where $\h$ is a certain family of vector fields on $\br^n$ isomorphic to 
the affine orthogonal Lie algebra, $n \geq 3$.  Consider the standard coordinates on $\br^n$ given by 
$(x_1, \, x_2, \, \ldots \, , x_n)$ with unit vector fields $\pari$ parallel to the $x_i$ axes.  Let
$\alpha_{ij} = x_i \parj - x_j \pari$.  Then $\{ \alpha_{ij} \}$, $1 \leq i < j \leq n$, is a vector space basis for a Lie 
algebra  isomorphic to ${\frak{so}}(n)$.  Let $J_n$ be the $\br$ vector space spanned by $\{ \pari \}$, $i = 1$, 2, $
\ldots \, ,$ $n$.   Then $J_n$ is an Abelian Lie algebra.  Let $\h$ be the Lie algebra with basis given by the union of
$\{ \alpha_{ij} \}$, $1 \leq i < j \leq n$, and $\{ \pari \}$, $i = 1$, 2, $\ldots \, ,$ $n$.  Then $J_n$ is an ideal of $\h$ 
and there is a short exact sequence of Lie algebras
$$  0 \longrightarrow J_n \longrightarrow \h \longrightarrow \frak{so}(n) \longrightarrow 0  $$
with $\h / J_n \simeq \son$.  We use results from invariant theory \cite{Biyogmam}, where $HL^* (\h ; \, \br)$ is computed, to 
help in the calculation of $HL^* (\h ; \, \h)$.  In low dimensions $HL^1 (\h ; \, \h)$ and $HL^2 (\h ; \, \h)$  are
generated by a ``metric class"  $I : \h \to \h$ and an ``area class" $\rho : \h^{\otimes 2} \to \h$ respectively, where
\begin{align*}
& I \Big( \pari \Big) = \pari, \ \ \ I (\alpha_{ij}) = 0, \\
& \rho \Big( \pari \otimes \parj \Big) = \alpha_{ij}, \ \ \ \rho (g_1 \otimes g_2) = 0 \ \ {\rm{if}}  \ \ g_i \in \son \, .
\end{align*}
Higher dimensions of $HL^* (\h ; \, \h )$ contain echoes of these classes against a tensor algebra.  Let 
$\alpha^*_{ij}$ be the dual of $\alpha_{ij}$ with respect to the basis $\{ \alpha_{ij} \} \cup \{ \pari \}$ of
$\h$, and let $dx^i$ be the dual of $\pari$.  Let
$$  \gamma^*_n = \sum_{1 \leq i < j \leq n} (-1)^{i+j-1} \alpha^*_{ij} \otimes dx^1 \wedge  dx^2 \wedge 
\ldots \widehat{d}x^i \ldots \widehat{d}x^j \ldots \wedge dx^n \, . $$
The element $\gamma^*_n : \h \otimes J_n^{\wedge (n-2)} \to \br$ may be viewed as a factored volume form, which 
by skew-symmetry extends to $\h \otimes \h^{\otimes (n-2)} \to \br$.  In subsection $\S$ 3.3 we prove that
$$  HL^* (\h ; \, \h ) \simeq \langle I, \, \rho \rangle \otimes T( \gamma^*_n ),  $$
where $\langle I, \, \rho \rangle$ is the real vector space with basis $\{ I, \, \rho \}$ and
$T( \gamma^*_n ) := \sum_{q \geq 0} \langle \gamma^*_n \rangle ^{\otimes q}$ is the tensor algebra on the class
of $\gamma^*_n$.  

\subsection{The Pirashvili Spectral Sequence}

In this subsection we outline the cohomological version of the Pirashvili spectral sequence 
\cite{Pirashvili} for computing the Leibniz cohomology of a Lie algebra $\frak{g}$ with coefficients in a 
$\frak{g}$-module $V$, denoted $HL^*(\frak{g}; \, V)$.  Suppose that both $\frak{g}$ and $V$ are $k$-modules
for a commutative ring $k$.  We use the convention that $\frak{g}$ acts on the left of $V$, meaning that
for $g$, $h \in \frak{g}$ and $x \in V$, we have $g(hx) - h(gx) = [g, \, h](x)$.  Then $HL^*(\frak{g}; \, V)$ is the 
cohomology of the cochain complex $CL^n(\frak{g}; \, V) = {\rm{Hom}}_k ( \frak{g}^{\otimes n}, \, V)$, $n = 0$, 
1, 2, 3, $\ldots$ .  For $f \in CL^n(\frak{g}; V)$, $\delta f \in CL^{n+1}(\frak{g}; \, V)$ is given by Equation
\eqref{coboundary}.  If in Equation \eqref{coboundary} the tensor product is replaced with the exterior product, then
we have a cochain complex for $H^*_{\rm{Lie}}(\frak{g}; \, V)$, the Lie-algebra cohomology of $\frak{g}$ with coefficients
in $V$.  Using Pirashvili's grading, let $\pi_{\rm{rel}} : \frak{g}^{\otimes (n+2)} \to \frak{g}^{\wedge (n+2)}$, $n \geq 0$,
be the projection
$$  \pi_{\rm{rel}}(g_1 \otimes g_2 \otimes \ldots \otimes g_{n+2}) = g_1 \wedge g_2 \wedge \ldots \wedge g_{n+2} . $$
There is an induced map on cohomology $\pi^*_{\rm{rel}} : H^*_{\rm{Lie}}(\frak{g}; \, V) \to HL^*(\frak{g}; \, V)$, and we 
study the relative groups.  Let 
$$ C^n_{\rm{rel}} = {\rm{Coker}}\big[ {\rm{Hom}}(\frak{g}^{\wedge (n+2)}, \, V ) 
\overset{\pi_{\rm{rel}}\; }{\longrightarrow} {\rm{Hom}}(\frak{g}^{\otimes (n+2)}, \, V ) \big], \ \ \ 
n = 0, \ 1, \ 2, \ \ldots \ . $$ 
Then $C^*_{\rm{rel}}$ is a cochain complex with cohomology denoted by $H^*_{\rm{rel}}(\frak{g}; \, V)$.  
It follows quickly that $H^0_{\rm{Lie}}(\frak{g}; V) \simeq HL^0( \frak{g}; \, V)$ and 
$H^1_{\rm{Lie}}(\frak{g}; V) \simeq HL^1( \frak{g}; \, V)$.  For higher dimensions there is a 
``Lie-to-Leibniz" long exact sequence:
\begin{equation} \label{H-rel}
\begin{split}
0 & \lra H^2_{\rm{Lie}} (\frak{g}; \, V) \overset{\pi^*_{\rm{rel \;}}}{\lra} HL^2(\frak{g}; \, V) \lra H^0_{\rm{rel}}(\frak{g}; \, V)
\overset{c_{\rm{rel \;}}}{\lra} H^3_{\rm{Lie}} (\frak{g}; \, V) \lra  \ldots \\
& \lra H^n_{\rm{rel}}( \frak{g}; \, V) \overset{c_{\rm{rel \;}}}{\lra} H^{n+3}_{\rm{Lie}} (\frak{g}; \, V)
\overset{\pi^*_{\rm{rel \;}}}{\lra} HL^{n+3}(\frak{g}; \, V) \lra H^{n+1}_{\rm{rel}}(\frak{g}; \, V) \lra \ldots \ .
\end{split}
\end{equation}
Above, $c_{\rm{rel}}$ is the connecting homomorphism.  The Pirashvili spectral sequence arises from a filtration  of 
$C^*_{\rm{rel}}$ and converges to $H^*_{\rm{rel}}(\frak{g}; \, V)$.  

First we describe $H^*_{\rm{Lie}}(\frak{g}; \, \frak{g}')$, the Lie-algebra cohomology of $\frak{g}$ with coefficients in 
the coadjoint representation $\frak{g}' = {\rm{Hom}}_k (\frak{g}, \, k)$.  For $\varphi \in \frak{g}'$ and $g \in \frak{g}$, 
the left action of $\frak{g}$ on $\frak{g}'$ is given by 
$$  (g \varphi ) (x) =  \varphi ([x, \, g]), \ \ \ x \in \frak{g} .  $$
Then $H^*_{\rm{Lie}}( \frak{g}; \, \frak{g}')$ is the cohomology of the complex 
${\rm{Hom}}_k ( \frak{g}^{\wedge *}, \, \frak{g}')$.
In general $H^*_{\rm{Lie}}( \frak{g} ; \, \frak{g}')$ is not isomorphic to $H^*_{\rm{Lie}}( \frak{g} ; \, \frak{g})$.
However, $H^*_{\rm{Lie}}( \frak{g}; \, \frak{g}')$ can be computed from the isomorphic complex
${\rm{Hom}}_k ( \frak{g} \otimes \frak{g}^{ \wedge *}, \, k)$, which we now describe.  Define 
$d : \frak{g} \otimes \frak{g}^{\wedge (n+1)} \to \frak{g} \otimes \frak{g}^{\wedge n} $, $n = 0$, 1, 2, $\ldots$ , by
\begin{equation*}
\begin{split}
 d & (x \otimes g_1 \wedge g_2 \wedge \ldots \wedge g_{n+1}) = 
 \sum_{i=1}^{n+1} (-1)^{i+1} [x, \, g_i] \otimes g_1 \wedge \ldots \widehat{g}_i \ldots \wedge g_{n+1} \\
 & + \sum_{1 \leq i < j \leq n+1} (-1)^{j+1} x \otimes g_1 \wedge \ldots g_{i-1} \wedge [g_i , \, g_j] \wedge
 g_{i+1} \ldots  \widehat{g}_j  \ldots \wedge g_{n+1} .  
\end{split}
\end{equation*}
The Lie-algebra {\em homology} groups $H^{\rm{Lie}}_* (\frak{g}; \, \frak{g})$ are computed from the complex
$(\frak{g} \otimes \frak{g}^{\wedge *}, \, d)$.  Let
$$  d^* : {\rm{Hom}}_k ( \frak{g} \otimes \frak{g}^{ \wedge n}, \, k) \to 
{\rm{Hom}}_k ( \frak{g} \otimes \frak{g}^{ \wedge (n +1)}, \, k)  $$
be the ${\rm{Hom}}_k$ dual of $d$, and let $\alpha \in {\rm{Hom}}_k (\frak{g}^{\wedge n}, \, \frak{g}')$.  There
is a cochain isomorphism
$$  \Phi : {\rm{Hom}}_k ( \frak{g}^{\wedge n}, \, \frak{g}') \to
{\rm{Hom}}_k ( \frak{g} \otimes \frak{g}^{\wedge n}, \, k)  $$ 
given by
$$ \Phi (\alpha) (x \otimes g_1 \wedge g_2 \wedge \ldots \wedge g_n) =
(-1)^n \alpha (g_1 \wedge g_2 \wedge \ldots \wedge g_n) (x),  $$
where $x$, $g_1$, $g_2$, $\ldots \, $, $g_n \in \frak{g}$.  The reader may verify that 
$\Phi (\delta \alpha) = d^* ( \Phi \alpha)$.  
\begin{lemma} \label{univ-coeff-thm}
When $k$ is a field $\bf{F}$, we have
$$  H^*_{\rm{Lie}}(\frak{g}; \, \frak{g}') \simeq 
{\rm{Hom}}_{\bf{F}} \big( H^{\rm{Lie}}_*( \frak{g}; \, \frak{g}), \, \bf{F} \big) . $$
\end{lemma}
\begin{proof}
The proof follows from the existence of $\Phi$ and the universal coefficient theorem.
\end{proof}

Consider now the projection
\begin{equation*}
\begin{split}
&  \pi_R : \frak{g} \otimes \frak{g}^{\wedge (n+1)} \to \frak{g}^{\wedge (n+2)}, \ \ \ n \geq 0 , \\
& \pi_R (g_0 \otimes g_1 \wedge g_2 \wedge \ldots \wedge g_{n+1}) =
g_0 \wedge g_1 \wedge g_2 \wedge \ldots \wedge g_{n+1} .
\end{split}
\end{equation*}
We work over a field $\bf{F}$, although many constructions are valid for an arbitrary commutative ring.  
Let
$$ CR^n  (\frak{g}) = {\rm{Coker}}\big[ {\rm{Hom}}_{\bf{F}}(\frak{g}^{\wedge (n+2)}, \, {\bf{F}} ) 
\overset{\pi^*_{R}\; }{\longrightarrow} {\rm{Hom}}_{\bf{F}}(\frak{g} \otimes \frak{g}^{\wedge (n+1)}, {\bf{F}} ) \big].  $$
Then $CR^*( \frak{g} )$ is a cochain complex whose cohomology is denoted $HR^* ( \frak{g} )$.  We have
$H^1_{\rm{Lie}}( \frak{g}; \, {\bf{F}} )\simeq H^0_{\rm{Lie}} ( \frak{g}; \, \frak{g}' )$, and there is a
``Lie-coadjoint" long exact sequence:  
\begin{equation} \label{HR-exact}
\begin{split}
0 & \lra H^2_{\rm{Lie}} (\frak{g}; \, {\bf{F}}) \overset{\pi^*_R}{\lra} H^1_{\rm{Lie}} (\frak{g}; \, \frak{g}') \lra HR^0 (\frak{g})
\overset{c_R}{\lra} H^3_{\rm{Lie}} (\frak{g}; \, {\bf{F}} ) \lra  \ldots \\
& \lra HR^m ( \frak{g} ) \overset{c_R}{\lra} H^{m+3}_{\rm{Lie}} (\frak{g}; \, {\bf{F}})
\overset{\pi^*_R}{\lra} H^{m+2}_{\rm{Lie}}(\frak{g}; \, \frak{g}' ) \lra HR^{m+1}( \frak{g} ) \lra \ldots \ .
\end{split}
\end{equation}
where $c_R$ is the connecting homomorphism.  We are now ready to state the version of the Pirashvili spectral 
sequence used in this paper.  Compare with \cite{Lodder} and \cite{Pirashvili}.

\begin{theorem}
Let $\frak{g}$ be a Lie algebra over a field $\bf{F}$ and let $V$ be a left $\frak{g}$-module.  Then there is a 
first-quadrant spectral sequence converging to $H^*_{\rm{rel}}( \frak{g} ; \, V)$ with
$$  E_2^{m, \, k} \simeq HR^m ( \frak{g} ) \otimes HL^k ( \frak{g}; \, V), \ \ \ m \geq 0, \ \ \ k \geq 0 ,  $$
provided that $HR^m ( \frak{g} )$ and $HL^k ( \frak{g} ; \, V)$ are finite dimensional vector spaces in each dimension.  
If this finiteness condition is not satisfied, then the completed tensor product $\widehat {\otimes}$ can be used.
\end{theorem}
\begin{proof}
We outline the key features of the construction and introduce notation that will be used in the sequel.  Let
$A^{m, \, k}$ denote those elements $f \in {\rm{Hom}}( \frak{g}^{\otimes (k + m + 2)}, \; V)$ that are skew-symmetric in the last
$m+1$ tensor factors of $\frak{g}^{\otimes (k + m + 2)}$.  Filter the complex $C^*_{\rm{rel}}$ by 
$$  F^{m, \, k} =   A^{m, \, k} / {\rm{Hom}}( \Lambda^{k+m+2} ( \frak{g}) , \, V) .  $$
Then $F^{m, \, *}$ is a subcomplex of $C^*_{\rm{rel}}$ with $F^{0, \, *} = C^*_{\rm{rel}}$ and  
$F^{m+1, \, *}  \subseteq F^{m, \, *}$.  To identify the $E^{*, \, *}_0$ term, use the isomorphism
\begin{equation}  \label{switch}
{\rm{Hom}}( \frak{g}^{\otimes (k+m+2)}, \, V) =
 {\rm{Hom}}( \frak{g}^{\otimes k} \otimes \frak{g}^{\otimes (m+2)}, \, V) \simeq  
 {\rm{Hom}}( \frak{g}^{\otimes (m+2)}, \, {\rm{Hom}} ( \frak{g}^{ \otimes k}, \, V)) 
 \end{equation}
Then 
\begin{equation*}
\begin{split}
E^{m, \, k}_0 & = F^{m, \, k}/ F^{m+1, \, k-1}  \\
& \simeq {\rm{Hom}} ( \frak{g} \otimes \frak{g}^{\wedge (m+1)} / \frak{g}^{\wedge (m+2)} , \, 
{\rm{Hom}}( \frak{g}^{\otimes k}, \, V) ), 
\end{split}
\end{equation*}
and $d^0_{m, \, k} : E^{m, \, k}_0 \to E^{m, \, k+1}_0$, $m \geq 0$, $k \geq 0$.  It follows that
$$  E^{m, \, k}_1 \simeq {\rm{Hom}}( \frak{g} \otimes \frak{g}^{\wedge (m+1)} / \frak{g}^{\wedge (m+2)} , \, 
HL^k ( \frak{g}; \, V)) .  $$
Now, $d^1_{m, \, k} : E^{m, \, k}_1 \to E^{m+1, \, k}$.  Since the action of $\frak{g}$ on $HL^* ( \frak{g}; \, V)$
is trival, we have $E^{m, \, k}_2 \simeq HR^m ( \frak{g} ) \widehat{\otimes} HL^k ( \frak{g}; \, V)$.  Using
the isomorphism \eqref{switch}, we consider an element of $E^{m, \, k}_2$ operationally in the form
$HL^k ( \frak{g}; \, V) \widehat{\otimes} HR^m ( \frak{g} )$.  
\end{proof}

\subsection{Lie-Algebra Cohomology}

For a Lie algebra $\frak{g}$ and a left $\frak{g}$-module $V$ over a field, let
$$  V^{\frak{g}} = \{ v \in V \ | \ gv = 0 , \ \forall \, g \in \frak{g}  \}  $$
be the subspace  of invariants.  Let $V_{\frak{g}} = V/[ \frak{g}, \, V]$ be the quotient space of coinvariants, where
$[ \frak{g}, \, V]$ is the span of all elements of the form $\{ gv \ | \ g \in \frak{g}, \, v \in V \}$.  For 
$f \in  {\rm{Hom}}( \frak{g}^{\otimes n}, \, V)$ and $g \in \frak{g}$, define the action of $\frak{g}$ on 
$ {\rm{Hom}}( \frak{g}^{\otimes n}, \, V)$ by
\begin{equation*}
\begin{split}
(gf) & (x_1 \otimes x_2 \otimes \ldots \otimes x_n) = g \cdot f(x_1 \otimes x_2 \otimes \ldots \otimes x_n) \\
& + \sum_{i=1}^n f (x_1 \otimes \ldots \otimes x_{i-1} \otimes [x_i, \, g] \otimes x_{i+1} \otimes \ldots \otimes x_n),
\ \ \ x_i \in \frak{g} .
\end{split}
\end{equation*}
The action of $\frak{g}$ on $\bf{R}$ is always considered trivial, meaning $g \cdot c = 0$, $\forall \, g \in \frak{g}$, 
$\forall \, c \in \bf{R}$,
which determines the action of $\frak{g}$ on ${\rm{Hom}}( \frak{g}^{\otimes n}, \, \br )$.  
The above action on ${\rm{Hom}}( \frak{g}^{\otimes n}, \, V)$ clearly induces an action on skew-symmetric 
elements of ${\rm{Hom}}( \frak{g}^{\otimes n}, \, V)$.  The resulting action of $\frak{g}$ on
${\rm{Hom}}( \frak{g}^{\wedge *}, \, V )$ commutes with the (Chevalley-Eilenberg) coboundary map $\delta$, 
and induces an action on $H^*_{\rm{Lie}}( \frak{g}; \, V)$ \cite{Hoch-Serre}.  Recall the definition of
$\h$ and $J_n$ given at the beginning of $\S$ 3.  
In this subsection we compute
$H^*_{\rm{Lie}}( \frak{h}_n ; \, \br)$, $H^*_{\rm{Lie}}(\frak{h}_n ; \, \frak{h}_n)$, and 
$H^*_{\rm{Lie}}(\frak{h}_n ; \, \frak{h}'_n)$ by applying the Hochschild-Serre spectral sequence \cite{Hoch-Serre}
and identifying certain $\son$-invariant cochains.  The calculations for $HR^*( \frak{h}_n )$ follow from exact sequence 
\eqref{HR-exact}.  
\begin{lemma}
For $n \geq 3$, there are Hochschild-Serre isomorphisms in Lie-algebra cohomology
\begin{align*}
&  H^*_{\rm{Lie}} (  \frak{h}_n ; \, \br ) \simeq 
H^*_{\rm{Lie}} ( \son ; \,  \br ) \otimes \big[ H^*_{\rm{Lie}} ( J_n; \, \br ) \big]^{\son} \\
&  H^*_{\rm{Lie}} (  \frak{h}_n ; \, \frak{h}_n ) \simeq 
H^*_{\rm{Lie}} ( \son ; \, \br ) \otimes \big[ H^*_{\rm{Lie}} ( J_n; \, \frak{h}_n ) \big]^{\son} \\
&  H^*_{\rm{Lie}} (  \frak{h}_n ; \, \frak{h}'_n ) \simeq 
H^*_{\rm{Lie}} ( \son ; \, \br ) \otimes \big[ H^*_{\rm{Lie}} ( J_n; \, \frak{h}'_n ) \big]^{\son} 
\end{align*}
\end{lemma}
\begin{proof}
The proof follows by applying the Hochschild-Serre spectral sequence \cite{Hoch-Serre} to the ideal $J_n$ of
$\frak{h}_n$ and using the isomorphism of Lie algebras $\frak{h}_n / J_n \simeq \son$.
\end{proof}

Now $[ H^*_{\rm{Lie}} ( J_n; \, \br ) ]^{\son}$  is the cohomology of the cochain complex
$$  \big[ {\rm{Hom}} ( J^{\wedge *}_n , \, \br ) \big]^{\son}  \simeq {\rm{Hom}} ( (J^{\wedge *}_n)_{\son} , \, \br ) \simeq
{\rm{Hom}}( (J^{\wedge *}_n)^{\son} , \, \br ) .  $$ 
From \cite{Biyogmam} or by a direct calculation
$$  (J^{\wedge k}_n )^{\son} =  \begin{cases} \br  &  k = 0 \\
\langle \tfrac{\partial}{\partial x^1} \wedge \tfrac{\partial}{\partial x^2 } \wedge
\ldots \wedge	\tfrac{\partial}{\partial x^n }\rangle &  k = n \\			
0  &  \mbox{otherwise.}    \end{cases}	$$		
Let $dx^i$ be the element of ${\rm{Hom}}( \frak{h}_n , \, \br )$ dual to $\frac{\partial}{\partial x^i}$ with respect to
the basis $\{ \alpha_{ij} \} \cup \{ \frac{\partial}{\partial x^i} \}$ of $\frak{h}_n$.  
\begin{lemma}
For $n \geq 3$, 
$$  H^*_{\rm{Lie}} ( \frak{h}_n ; \, \br ) \simeq H^*_{\rm{Lie}}( \son ; \, \br ) \oplus
\big(  H^*_{\rm{Lie}}( \son ; \, \br ) \otimes \langle v^*_n \rangle \big),  $$ 
where $v^*_n = dx^1 \wedge dx^2 \wedge \ldots \wedge  dx^n$.  
\end{lemma}
Likewise, $[ H^*_{\rm{Lie}} ( J_n ; \, \frak{h}'_n ) ]^{\son}$ is the cohomology of the cochain complex
$$  \big[ {\rm{Hom}} ( \frak{h}_n \otimes J^{\wedge *}_n , \, \br ) \big]^{\son}  
\simeq {\rm{Hom}} ( ( \frak{h}_n \otimes J^{\wedge *}_n )_{\son} , \, \br ) \simeq
{\rm{Hom}}( ( \frak{h}_n \otimes J^{\wedge *}_n )^{\son} , \, \br ) .  $$ 
Now, $(\frak{h}_n \otimes J^{\wedge *}_n)^{\son} \simeq (J_n \otimes J^{\wedge *}_n )^{\son} \oplus
( \son \otimes J^{\wedge *}_n )^{\son}$.  From \cite{Biyogmam}, we have
$$  ( J_n \otimes J^{\wedge k}_n )^{\son} \simeq \begin{cases}  \langle g_n \rangle &  k = 1 \\
	\langle w_n \rangle &  k = n - 1 \\
	0 & \mbox{otherwise,}	\end{cases}  $$
where $g_n = \sum_{i = 1}^n \frac{\partial}{\partial x^i} \otimes \frac{\partial}{\partial x^i}$, and
$$  w_n = \sum_{i = 1}^n (-1)^{i-1} \tfrac{\partial}{\partial x^i} \otimes 
\tfrac{\partial}{\partial x^1} \wedge \tfrac{\partial}{\partial x^2} \wedge \ldots  
\widehat{\tfrac{\partial}{\partial x^i}}  \ldots \wedge \tfrac{\partial}{\partial x^n} .  $$
Let $g^*_n = \sum_{i = 1}^n dx^i \otimes dx^i$, 
$w^*_n = \sum_{i = 1}^n (-1)^{n-1} dx^i \otimes dx^1 \wedge dx^2 \wedge \ldots \widehat{d}x^i \ldots  \wedge dx^n$.
Also from \cite{Biyogmam}
$$  ( \son \otimes J^{\wedge k}_n ) ^{\son} \simeq \begin{cases} \langle s_n \rangle &  k = 2 \\
	\langle \gamma_n \rangle &  k = n - 2  \\
 	 0  &  \mbox{otherwise,}  \end{cases}  $$ 
where $s_n = \sum_{1 \leq i < j \leq n} \alpha_{ij}  \otimes \frac{\partial}{\partial x^i} \wedge \frac{\partial}{\partial x^j}$,
and
$$  \gamma_n = \sum_{1 \leq i < j \leq n} (-1)^{i + j - 1} \alpha_{ij} \otimes \tfrac{\partial}{\partial x^1}	 
\wedge \tfrac{\partial}{\partial x^2}
\wedge \ldots \widehat{ \tfrac{\partial}{\partial x^i} } \ldots  \widehat{ \tfrac{\partial}{\partial x^j}  } \ldots
\wedge \tfrac{\partial}{\partial x^n}.  $$
Let $s^*_n = \sum_{1 \leq i < j \leq n} \alpha^*_{ij} \otimes dx^i \wedge dx^j$, and
$$  \gamma^*_n = \sum_{1 \leq i < j \leq n} \alpha^*_{ij} \otimes dx^1 \wedge dx^2 \ldots \widehat{d}x^i \ldots
\widehat{d}x^j \ldots \wedge dx^n .  $$
In the cochain complex $( {\rm{Hom}}( (\frak{h}_n \otimes J^*_n )^{\son}, \,  \br), \ \delta)$, we have
$\delta (g^*_n) = -2 s^*_n$, $\delta (\gamma^*_n ) = 0$, $\delta (w^*_n) = 0$.  
\begin{lemma}
For $n \geq 3$,
$H^*_{\rm{Lie}}(J_n; \, \frak{h}'_n ) \simeq H^*_{\rm{Lie}} ( \son ; \, \br ) \otimes 
\langle \gamma^*_n , \, w^*_n \rangle$.  
\end{lemma}
In the long exact sequence \eqref{HR-exact}, the map $\pi^*_R : H^*_{\rm{Lie}}( \frak{h}_n ; \, \br) \to
H^*_{\rm{Lie}}( \frak{h}_n ; \, \frak{h}'_n )$ sends $H^*_{\rm{Lie}} ( \son ; \br ) \otimes \langle v^*_n \rangle$
to $H^*_{\rm{Lie}} ( \son ; \br ) \otimes \langle w^*_n \rangle$.  It follows that
\begin{lemma}  \label{HR-calculation}
For $m \geq 0$, 
$$  HR^m( \frak{h}_n ) \simeq H^{m+3}_{\rm{Lie}} ( \son ; \, \br ) 
\oplus \big( H^{m+3-n}_{\rm{Lie}}( \son ; \, \br) \otimes \langle \gamma^*_n \rangle \big) , $$
where $c_R [ HR^m( \frak{h}_n ) ] = H_{\rm{Lie}}^{m+3} ( \son ; \br )$ and 
$$  H^{m+3-n} _{\rm{Lie}} ( \son ; \, \br ) \otimes \langle \gamma^*_n \rangle \subseteq
{\rm{Im}} [ H^{m+1}_{\rm{Lie}} ( \fhn ; \,  \fhn') \to HR^m( \fhn ) ] $$
in exact sequence \eqref{HR-exact}.  
\end{lemma}

To compute $H^*_{\rm{Lie}} ( \fhn ; \, \fhn )$, we must first identify 
$[ {\rm{Hom}} ( J^{\wedge *}_n , \, \fhn ) ]^{\son}$.   
\begin{lemma}  \label{iso-invar}
Let $\varphi \in {\rm{Hom}} ( J^{\wedge k}_n , \, \fhn )$ and 
$z = \frac{\partial}{\partial x^{i_1}} \wedge \frac{\partial}{\partial x^{i_2}}
\wedge \ldots \wedge \frac{\partial}{\partial x^{i_k}}$.  There is an $\son$-equivariant isomorphism
$$  \Phi : {\rm{Hom}} ( J^{\wedge k}_n , \, \fhn ) \to J^{\wedge k}_n \otimes \fhn \ \  \mbox{given by} \ \ 
\Phi (\varphi) = \sum_{i_1 < i_2 < \ldots < i_k} z \otimes \varphi (z) .  $$
\end{lemma}
\begin{proof}
The proof is a simple calculation.
\end{proof}    
Thus, $[ {\rm{Hom}} ( J^{\wedge k}_n , \, \fhn ) ]^{\son} \simeq ( J^{\wedge k}_n \otimes \fhn )^{\son}
\simeq ( \fhn \otimes J^{\wedge k}_n )^{\son}$, and the latter are discussed earlier in this section.  

By identification of invariants under $\Phi$, we have
\begin{lemma}
The invariants of $[ {\rm{Hom}} (J^{\wedge *}_n, \, \fhn ) ]^{\son}$, $n \geq 3$, are given by
\begin{equation*}
\begin{split}
& I \in {\rm{Hom}} (J_n, \, \fhn ), \ \ \rho \in {\rm{Hom}} (J^{\wedge 2}_n, \, \fhn ), \\ 
& \Gamma \in {\rm{Hom}} (J^{\wedge (n-2)}_n, \, \fhn ), \ \ \mu \in  {\rm{Hom}} (J^{\wedge (n-1)}_n, \, \fhn ),
\end{split}
\end{equation*}
where $I ( \frac{\partial}{\partial x^i} ) = \frac{\partial}{\partial x^i}$, $i = 1, \,  2, \,  \ldots \, $, $n$, 
$\rho  ( \frac{\partial}{\partial x^i} \wedge \frac{\partial}{\partial x^j} ) = \alpha_{ij}$
\begin{equation*}
\begin{split}
& \Gamma  ( \tfrac{\p}{\p x^1} \wedge  \tfrac{\p}{\p x^2} \wedge \ldots 
\widehat{ \tfrac{\p}{\p x^i} } \ldots \widehat{ \tfrac{\p}{\p x^j} }
\ldots \wedge \tfrac{\p}{\p x^n} ) = (-1)^{i+j-1} \alpha_{ij} \\ 
& \mu ( \tfrac{\p}{\p x^1} \wedge  \tfrac{\p}{\p x^2} \wedge \ldots 
\widehat{ \tfrac{\p}{\p x^j} } \ldots \wedge \tfrac{\p}{\p x^n} ) = (-1)^{j-1} \tfrac{\p}{\p x^j} .  
\end{split}
\end{equation*}
\end{lemma}
Now, $\delta I = 0$, $\delta \rho = 0$, and $\delta \Gamma  = (n - 1) (-1)^{n-1} \mu$.  

\begin{lemma}
For $n \geq 3$, $H^*_{\rm{Lie}} ( \fhn ; \, \fhn ) \simeq \langle I, \, \rho \rangle \otimes H^*_{\rm{Lie}} ( \son ; \, \br)$.
\end{lemma}

Note that $I : J_n \to \fhn$ may be extended to $I : \fhn \to \fhn$ by requiring that $I ( \alpha_{ij} ) = 0$.  
Also, $\rho : J^{\wedge 2}_n \to \fhn$ may be extended to $\rho : \fhn^{\wedge 2} \to \fhn$ by 
requiring that $\rho (g_1 \wedge g_2) = 0$ if either $g_i \in \son$.  Any element $\theta \in H^k_{\rm{Lie}}( \son ; \br )$
is represented by an $\son$-invariant cocycle $ \theta : (\son )^{ \wedge k} \to \br$ that can be extended
to $\theta  : \fhn^{\wedge k} \to \br$ by requiring that $\theta (h_1 \wedge h_2 \wedge \ldots \wedge h_n) = 0$
if any $h_i \in J_n$.  Under the isomorphism of the Hochschild-Serre spectral sequence, $I \otimes \theta$ and
$\rho \otimes \theta$ correspond to the cocycles $ I \wedge \theta : \fhn^{\wedge (k+1)} \to \fhn$ and
$\rho  \wedge \theta : \fhn^{\wedge (k+2)} \to \fhn$ respectively, where $I \wedge \theta$ and $\rho \wedge \theta$ 
represent the shuffle product.  For completeness, we close this section with a statement about 
$H^*_{\rm{Lie}} (\son ; \, \br)$, although the specific elements in $H^*_{\rm{Lie}} (\son ; \, \br)$ do not survive to
$HL^* ( \fhn ; \, \fhn )$.  From \cite{Ito}, we have

\begin{theorem}
Let $n \geq 3$.   For $n$ odd,
$$  H^*_{\rm{Lie}} ( \son ; \, \br ) \simeq \Lambda [ \{ x_{4i-1} \, | \ 0 < 2i < n \} ].  $$
For $n$ even,
$$  H^*_{\rm{Lie}} ( \son ; \, \br ) \simeq \Lambda [ \{ x_{4i-1}\, | \ 0 < 2i < n \} ] \otimes \Lambda [ \{ y_{n-1} \} ].  $$
\end{theorem}  

\subsection{Leibniz Cohomology with Adjoint Coefficients}

In this subsection we compute $HL^* ( \fhn ; \, \fhn )$ for the affine extension $\fhn$ of $\son$, $n \geq 3$.
In low dimensions, we have
\begin{align*}
& HL^0 (\fhn ; \, \fhn ) \simeq H^0_{\rm{Lie}} ( \fhn ; \, \fhn ) \simeq (\fhn )^{\son} \simeq \{ 0 \}  \\
& HL^1 (\fhn ; \, \fhn ) \simeq H^1_{\rm{Lie}} ( \fhn ; \, \fhn ) \simeq \langle I \rangle . \\
\end{align*}
where $I : \fhn \to \fhn$ is given above.  The calculations for $HL^* ( \fhn ; \, \fhn )$
proceed in a recursive manner, using results from lower dimensions to compute higher dimensions.
Since the $HR^m ( \fhn )$ groups are know by Lemma \eqref{HR-calculation}, information can be gleaned
about the Pirashvili spectral sequence with
$$  E^{m, \, k}_2 \simeq HR^m ( \fhn ) \otimes HL^k ( \fhn ; \, \fhn ), \ \ \ m \geq 0, \ \ k \geq 0 , $$
where previous results are substituted for $HL^k ( \fhn ; \, \fhn )$.
This spectral sequence converges to $H^*_{\rm{rel}} ( \fhn ; \, \fhn )$.  The $H^*_{\rm{rel}} ( \fhn ; \, \fhn )$
groups can then be inserted into the ``Lie-to-Leibniz" long exact sequence \eqref{H-rel} to determine the next
dimension(s) of $HL^* ( \fhn ; \, \fhn )$.  To illustrate this strategy in an easy example, note that
since 
$$  HL^0 ( \fhn ; \, \fhn) = 0,  $$ 
we have $E^{0, \, 0}_2 = 0$.  (In fact, $E^{m, \, 0}_2 = 0$, $m \geq 0$).  Thus,
$H^0_{\rm{rel}} ( \fhn ; \, \fhn) = 0$, and
$$  H^2_{\rm{Lie}} ( \fhn ; \, \fhn) \overset{\pi^*_{\rm{rel}}}{\longrightarrow} HL^2 ( \fhn ; \, \fhn )  $$
is an isomorphism.  In fact, $HL^2 ( \fhn ; \, \fhn)$ is generated by the class of
$$  \fhn \otimes \fhn \overset{\pi}{\longrightarrow} \fhn \wedge \fhn \overset{\rho}{\longrightarrow} \fhn ,  $$
which we simply denote as $\rho : \fhn^{\otimes 2} \to \fhn$
with $\rho ( \frac{\p}{\p x^i} \otimes \frac{\p}{\p x^j} ) = \alpha_{ij}$
for $i = 1$, 2, $\dots \,$, $n$, and $j = 1$, 2, $\ldots \, $, $n$.  Note that $\alpha_{ij} = - \alpha_{ji}$.  

\begin{lemma} \label{HL-0-n+1}
For $n \geq 3$, $0 \leq * \leq n + 1$, we have
$HL^* ( \fhn ; \, \fhn ) \simeq \langle I, \, \rho \rangle \otimes \big( \br \oplus \langle \gamma^*_n \rangle \big)$,
where 
$$ \gamma^*_n = \sum_{1 \leq i < j \leq n} (-1)^{i+j-1} \alpha^*_{ij} \otimes dx^1 \wedge dx^2 \wedge \ldots
\widehat{d}x^i \ldots \widehat{d}x^j \ldots \wedge dx^n .  $$
\end{lemma}
\begin{proof}
The results for $* = 0, \, 1, \, 2$ follow from above.  We consider the next iteration of elements in the $E_2^{*, \, *}$ 
term of the Pirashvili spectral sequence.  These are:
\begin{equation*}
\begin{split}
&  I \otimes \gamma^*_n \in HL^1 ( \fhn ; \, \fhn ) \otimes HR^{n-3} (\fhn ) \hookrightarrow E_2^{n-3, \, 1} \\
& \rho \otimes \gamma^*_n \in HL^2 ( \fhn ; \, \fhn ) \otimes HR^{n-3} (\fhn ) \hookrightarrow E_2^{n-3, \, 2} \\
& I \otimes \theta'  \in HL^1 ( \fhn ; \, \fhn ) \otimes HR^{m} (\fhn ) \hookrightarrow E_2^{m, \, 1} \\
& \rho \otimes \theta'  \in HL^2 ( \fhn ; \, \fhn ) \otimes HR^{m} (\fhn ) \hookrightarrow E_2^{m, \, 2} ,
\end{split}
\end{equation*}
where $c_R ( \theta' ) = \theta \in H_{\rm{Lie}}^{m+3} ( \son ; \, \br )$ in exact sequence \eqref{HR-exact} 
(the ``Lie-coadjoint" exact sequence).

Now,
\begin{equation*}
\begin{split}
\delta ( I \otimes \gamma^*_n ) & ( g_1 \otimes g_2 \otimes \ldots \otimes g_{n+1} ) = 
( \delta I ) \otimes \gamma^*_n - I \otimes ( \delta \g )  + \\
&  \sum_{i = 3}^{n+1} (-1)^i (g_i I) (g_1) \, \g ( g_2 \otimes g_3\otimes \ldots \widehat{g}_i \ldots  \otimes g_{n+1} ) .  
\end{split}
\end{equation*}
Note that $\delta I = 0$ and $\delta ( \g) = 0$.  
Since $I$ is an $\son$-invariant as well as an $J_n$-invariant, we have $g_i I = 0$ for all $g_i \in \fhn$.  
Thus, $\delta ( I \otimes \g ) = 0$, and $I \otimes \g$ is an absolute cocycle in both $C^{n-2} _{\rm{rel}}( \fhn ; \, \fhn )$
and $CL^n ( \fhn ; \, \fhn )$.  Since $c_{\rm{rel}} ( I \otimes \g ) = [ \delta ( I \otimes \g ) ] = 0$ in
$H^{n+1}_{\rm{Lie}} ( \fhn ; \, \fhn )$, $I \otimes \g$ represents an element in $HL^n ( \fhn ; \, \fhn )$.  
Also,
\begin{equation}  \label{delta-rho-g}
\begin{split}
\delta ( \rho \otimes \g ) & ( g_1 \otimes g_2 \otimes \ldots \otimes g_{n+2} ) =  
( \delta \rho ) \otimes \g + \rho \otimes ( \delta \g ) + \\
& \sum_{i = 4}^{n + 2} (-1)^i ( g_i \rho )( g_1 \otimes g_2 ) \, \g ( g_3 \otimes g_4 \otimes \ldots 
\widehat{g}_i \ldots \otimes g_{n+2} ).
\end{split}
\end{equation}
Clearly, $\delta \rho = 0$ and $\delta \g = 0$.  In the last summand in equation \eqref{delta-rho-g} above,
the $\g$ term is zero unless $g_3 \in \son$ and $g_i \in J_n$, $i \geq 4$.  Since 
$d^2( \rho \otimes \g ) \in F_{n-1}$, $d^2 ( \rho \otimes \g )$ is skew-symmetric in $g_3$ and $g_4$.  But
every term of $\g ( g_4 \otimes g_3 \otimes \ldots \widehat{g}_i \ldots g_{n+2} )$ would then be zero.  Thus,
$d^2 ( \rho \otimes \g ) = 0$.  The same argument applies to higher differentials such as
$d^3 ( \rho \otimes \g ) \in F_n$.  Thus $\rho \otimes \g$ corresponds to a class in 
$H^{n+1}_{\rm{rel}} (\fhn ; \, \fhn)$.  Skew-symmetrization considerations of $\g$ also lead to 
$c_{\rm{rel}} ( \rho \otimes \g ) = [ \delta ( \rho \otimes \g ) ] = 0$ in $H^{n+2}_{\rm{Lie}} ( \fhn ; \, \fhn )$.  
Thus, $\rho \otimes \g$ corresponds to a class in $HL^{n+1} ( \fhn ; \, \fhn )$.  

Consider $\theta' \in HR^m ( \fhn )$ with $c_R ( \theta' ) = \theta \in H^{m+3}_{\rm{Lie}} ( \son ; \, \br )$.  Then
$$  \delta ( I \otimes \theta' ) = - I \otimes  ( \delta \theta' ) = - I \otimes c_R ( \theta' ) = - I \otimes \theta .  $$
From the Hochschild-Serre spectral sequence, it follows that $[ I \otimes \theta ] = [ I \wedge \theta ]$ 
represents a non-zero class in $H^*_{\rm{Lie}} ( \fhn ;  \, \fhn )$.  Now, $\pi^*_{\rm{rel}} ( [ I \otimes \theta ] ) = 0$
in exact sequence \eqref{H-rel}, since $I \otimes \theta$ is a coboundary in $CL^* ( \fhn ; \, \fhn )$.  Thus,
$I \otimes \theta'$ represents a non-zero element in $H^*_{\rm{rel}} (  \fhn ;  \, \fhn )$ with
$$  c_{\rm{rel}} ( [ I \otimes \theta' ] ) = [ \delta ( I \otimes \theta' ) ] = [ - I \otimes \theta ] .  $$

Now,
\begin{equation}  \label{rho-theta'}
\begin{split}
\delta ( \rho \otimes \theta' ) & ( g_1 \otimes g_2 \otimes \ldots \otimes g_{m+5} ) = \rho \otimes ( \delta \theta') + \\
& \sum_{j=4}^{m+5} (-1)^j (g_i \rho )(g_1 \otimes g_2 ) \, \theta' ( g_3 \otimes g_4 \otimes \ldots \hat{g}_j \ldots
\otimes g_{m+5} )
\end{split}
\end{equation}
The element $\theta' \in HR^m ( \fhn )$ can be chosen so that 
$\theta' ( y_1 \otimes y_2 \otimes \ldots \otimes y_{m+2} ) = 0$ if any $y_i \in J_n$, while $\rho$ is an
$\son$-invariant.  Suppose then that $g_j \in J_n$ for one and only one $j \in \{ 4, \, 5, \, \ldots \, , m+5 \}$.  
We have  $\rho \otimes \theta' \in F_m$ and $d^2 ( \rho \otimes \theta' ) \in F_{m+2}$.  Clearly,
$\delta ( \theta' ) = \theta$ is already skew-symmetric in $g_3$, $g_4$, $\ldots \, $, $g_{m+5}$ in equation
\eqref{rho-theta'}.  Thus, for $d^2 ( \rho \otimes \theta' )$, we must have
\begin{equation*}
\begin{split}
(g_j \rho) & ( g_1 \otimes g_2 )  \, \theta' (g_3 \otimes g_4 \otimes \ldots \widehat{g}_j \ldots \otimes g_{m+5} ) = \\
& \pm (g_3 \rho ) ( g_1 \otimes g_2 )  \, \theta' (g_j \otimes g_4 \otimes \ldots \widehat{g}_j \ldots \otimes g_{m+5} ) = 0,
\end{split}
\end{equation*}
since $g_3 \in \son$ and $\rho$ is an $\son$-invariant.  Also,
$$  d^3 ( \rho \otimes \theta' ) \in HL^0 ( \fhn ; \, \fhn ) \otimes HR^{m+3}( \fhn ) = \{ 0 \} .  $$
Thus, $\rho \otimes \theta'$ represents an element in $H^{m+2}_{\rm{rel}} ( \fhn ; \, \fhn )$ and 
$$  c_{\rm{rel}}( [ \rho \otimes \theta' ] ) = [ \delta ( \rho \otimes \theta' ) ] = [ \rho \otimes \theta ] = [ \rho \wedge \theta ]  $$
in $H^{m+5}_{\rm{Lie}} ( \fhn ; \, \fhn )$, where $c_{\rm{rel}}$ is the connecting homomorphism 
the ``Lie-to-Leibniz" exact sequence.  

\end{proof}

\begin{theorem}
For $n \geq 3$, $HL^* ( \fhn ; \, \fhn ) \simeq \langle I , \,  \rho \rangle \otimes T( \gamma^*_n )$, where
$T( \gamma^*_n ) := \sum_{q \geq 0} \langle \gamma^*_n \rangle ^{\otimes q}$ is the tensor algebra on the class of
$\gamma^*_n$.
\end{theorem}
\begin{proof}
The result for $q = 0$ and $q = 1$ follows from Lemma \eqref{HL-0-n+1}.  We begin with the next iteration
of elements in the $E^{*, \, *}_2$ term of the Pirashvili spectral sequence.  
Consider
$$  ( I \otimes \g ) \otimes \theta' \in HL^n ( \fhn ; \, \fhn ) \otimes HR^m ( \fhn ) \subseteq E^{m, \, n}_2  $$
with $c_R ( \theta' ) = \theta \in H^{m+3}_{\rm{Lie}} ( \son ; \, \br )$.  Then
\begin{equation*}
\begin{split} 
 \delta  & \big( ( I \otimes \g ) \otimes \theta' \big) ( g_1 \otimes \ldots \otimes g_{n+m+3} ) =  
 (-1)^n ( I \otimes \g) \otimes ( \delta \theta' ) + \\
& \sum_{j = n+2}^{n+m+3} (-1)^j \big( g_j ( I \otimes \g ) \big) (g_1 \otimes  \ldots  \otimes g_n ) \,
\theta' ( g_{n+1} \otimes \ldots \widehat{g}_j \ldots \otimes g_{n+m+3} ) .  
\end{split}
\end{equation*} 
Now, $I$ is an $\fhn$-invariant, and it can be checked that $\g$ is as well.  Thus, $I \otimes \g$ is an $\fhn$-invariant
and $g_j ( I \otimes \g ) = 0$ for all $g_j \in \fhn$.  Thus,
\begin{equation*}
\begin{split}
\delta \big( ( I \otimes \g ) \otimes \theta' \big) & = (-1)^n ( I \otimes \g ) \otimes \theta \\
& = (-1)^n I \otimes ( \g \otimes \theta) \in HL^1 ( \fhn ; \, \fhn ) \otimes HR^{n+m} ( \fhn ) .
\end{split}
\end{equation*}  
In the Pirashvili spectral sequence, $d^n ( ( I \otimes \g ) \otimes \theta' ) = I \otimes (\g \otimes \theta )$.  

Now consider an $\son$-invariant representation for the cohomology class corresponding to
$[ \rho \otimes \g ] \in HL^{n+1} ( \fhn ; \, \fhn )$.  Let $\theta' \in HR^m (\fhn )$ with
$c_R ( \theta' ) = \theta \in H^{m+3}_{\rm{Lie}} ( \son ; \, \br )$.  By a skew symmetry argument as before, we have
$$  d^n \big( ( \rho \otimes \g ) \otimes \theta' \big) = \rho \otimes ( \g \otimes \theta )  $$
in the Pirashvili spectral sequence.  For
$$  [ \rho \otimes \g ] \otimes \g \in HL^{n+1} ( \fhn ; \, \fhn ) \otimes HR^{n-3} ( \fhn ) \subseteq E^{n-3, \, n+1}_2 ,  $$
we have
\begin{equation*}
\begin{split}
\delta ( [ \rho \otimes \g ] &  \otimes \g )  ( g_1 \otimes \ldots \otimes g_{2n+1}) = 
\delta ( [ \rho \otimes \g ] ) \otimes \g + (-1)^{n+1} [ \rho \otimes \g ] \otimes ( \delta \g ) + \\
& \sum_{j = n+3}^{2n + 1} \big( g_j [ \rho \otimes \g ] ( g_1 \otimes \ldots \otimes g_{n+1}) \big) \, 
\g ( g_{n+2} \otimes \ldots \widehat{g}_j \ldots \otimes g_{2n + 1} ) = \\
& \sum_{j = n+3}^{2n + 1} \big( g_j [ \rho \otimes \g ] ( g_1 \otimes \ldots \otimes g_{n+1} ) \big) \, 
\g ( g_{n+2} \otimes \ldots \widehat{g}_j \ldots \otimes g_{2n + 1} ).  
\end{split}
\end{equation*}
Since $d^2 ( [ \rho \otimes \g ] \otimes \g ) \in F_{n-1}$, $d^2 ( [ \rho \otimes \g ] \otimes \g )$ is skew-symmetric in 
the variables $g_{n+2}$, $g_{n+3}$, $\ldots \, $, $g_{2n +1}$.  By a similar argument used for $\rho \otimes \g$,
we have that $[ \rho \otimes \g ] \otimes \g$ corresponds to an element in $HL^{2n} ( \fhn ; \, \fhn )$.  

For $[ I \otimes \g ] \in HL^n ( \fhn ; \, \fhn )$ and $\g \in HR^{n-3} ( \fhn )$, we have
\begin{equation*}
\begin{split}
\delta ( [ I \otimes \g ] & \otimes \g )  ( g_1 \otimes \ldots \otimes g_{2n} ) = 
\delta ( I \otimes \g ) \otimes \g + (-1)^n ( I \otimes \g ) \otimes ( \delta \g ) + \\
& \sum_{j = n+2}^{2n} (-1)^j \big( g_j ( I \otimes \g ) \big) (g_1 \otimes \ldots \otimes g_n ) \,
\g (g_{n+1} \otimes  \ldots \widehat{g}_j \ldots g_{2n} ) = 0 ,
\end{split}
\end{equation*}
since $I \otimes \g$ is an $\fhn$-invariant.  Thus, $( I \otimes \g ) \otimes \g$ represents a class in
$HL^{2n -1 }( \fhn ;  \, \fhn )$.  The theorem holds for $q = 2$.  
Also, in the Pirashvili spectral sequence, we have
\begin{equation*}
\begin{split}
& d^n \big( ( I \otimes ( \g )^{\otimes 2} ) \otimes \theta' \big) = (I \otimes  \g ) \otimes ( \g \otimes \theta )\\
& d^n \big( ( \rho \otimes ( \g )^{\otimes 2} ) \otimes \theta' \big) = (\rho \otimes  \g ) \otimes ( \g \otimes \theta ). 
\end{split}
\end{equation*} 
By induction on $q$, $HL^* ( \fhn ; \, \fhn )$ is the direct sum of vector spaces
$\langle I , \, \rho \rangle \otimes ( \g )^{\otimes q}$.  We conclude that
$$  HL^* ( \fhn ; \, \fhn ) \simeq  \langle I , \, \rho \rangle \otimes T ( \g ) .  $$

\end{proof}


\end{document}